\documentclass[twoside]{article}

\usepackage{amsmath,amsthm,amssymb}
\usepackage{newtxtext,newtxmath}
\usepackage[text={12.5cm,19cm},centering,paperwidth=17cm,paperheight=24cm]{geometry}
\usepackage[fontsize=10.4pt]{scrextend}
\def\titlerunning#1{\gdef\titrun{#1}}

\makeatletter
\def\author#1{\gdef\autrun{\def\and{\unskip, }#1}\gdef\@author{#1}}
\def\address#1{{\def\and{\\\hspace*{15.6pt}}\renewcommand{\thefootnote}{}\footnote{#1}}\markboth{\autrun}{\titrun}}
\makeatother

\def\email#1{email: \href{mailto:#1}{#1} }
\def\subjclass#1{\par\bigskip\noindent\textbf{Mathematics Subject Classification 2020.} #1}
\def\keywords#1{\par\smallskip\noindent\textbf{Keywords.} #1}

\newenvironment{dedication}{\itshape\center}{\par\medskip}

\newtheorem{thm}{Theorem}[section]
\newtheorem{cor}[thm]{Corollary}
\newtheorem{prop}[thm]{Proposition}
\newtheorem{lem}[thm]{Lemma}

\theoremstyle{definition}

\numberwithin{equation}{section}

\newlength{\bibitemsep}
\newlength{\bibparskip}
\let\oldthebibliography\thebibliography
\renewcommand\thebibliography[1]{\oldthebibliography{#1}\setlength{\parskip}{\bibitemsep}\setlength{\itemsep}{\bibparskip}}

\usepackage[hyperfootnotes=false,colorlinks=true,allcolors=blue]{hyperref}

\newcommand{\R}{{\mathbb R}}
\newcommand{\N}{{\mathbb N}}
\newcommand{\Z}{{\mathbb Z}}
\newcommand{\C}{{\mathbb C}}
\newcommand{\be}[1]{\begin{equation}\label{#1}}
\newcommand{\ee}{\end{equation}}
\renewcommand{\(}{\left(}
\renewcommand{\)}{\right)}
\renewcommand{\pmatrix}[1]{\(\begin{matrix}#1\end{matrix}\)}
\newcommand{\A}{\mathbf A_a}
\newcommand{\nA}{\nabla_{\kern-2pt a}}
\newcommand{\dz}{\partial_z}
\newcommand{\dzbar}{\partial_{\bar z}}
\newcommand{\crit}{\mathsf c}
\newcommand{\acrit}{\mathsf a}
\newcommand{\D}{D_{\kern-1pt a}}

\begin{document}
\titlerunning{Critical magnetic field for 2d magnetic Dirac-Coulomb operators}
\title{\textbf{Critical magnetic field for\\ 2d magnetic Dirac-Coulomb operators\\ and Hardy inequalities}}
\author{Jean Dolbeault \and Maria J. Esteban \and Michael Loss}
\date{}

\maketitle
\thispagestyle{empty}

\address{J. Dolbeault, M. J. Esteban: CEREMADE (CNRS UMR n$^\circ$ 7534), PSL university, Universit\'e Paris-Dauphine,
Place de Lattre de Tassigny, 75775 Paris 16, France; \email{dolbeaul@ceremade.dauphine.fr, esteban@ceremade.dauphine.fr}}
\address{M. Loss: School of Mathematics, Skiles Building, Georgia Institute of Technology,
Atlanta GA 30332-0160, USA; \email{loss@math.gatech.edu}}

\begin{dedication}
It is with great pleasure that we dedicate this paper to Ari the occasion of his 70th birthday, because of his strong interest in Hardy inequalities and his pioneering work in this area.
\end{dedication}


\begin{abstract}

This paper is devoted to the study of the two-dimensional Dirac-Coulomb operator in presence of an Aharonov-Bohm external magnetic potential. We characterize the highest intensity of the magnetic field for which a two-dimensional magnetic Hardy inequality holds. Up to this critical magnetic field, the operator admits a distinguished self-adjoint extension and there is a notion of ground state energy, defined as the lowest eigenvalue in the gap of the continuous spectrum.




\end{abstract}

\section{Introduction and main results}\label{Sec:Introduction}

\emph{Aharonov-Bohm magnetic fields} describe an idealized situation where a solenoid interacts with a charged quantum mechanical particle and it is customary to take this to be a particle without spin. It was realized by Laptev and Weidl~\cite{Laptev-Weidl} that in such a situation a Hardy inequality holds, provided that the flux of the solenoid is not an integer multiple of $2 \pi$. This result was extended in~\cite{bonheureetal} to the non-linear case yielding a sharp Caffarelli-Kohn-Nirenberg type inequality. The current paper is devoted to \emph{Hardy inequalities} but for spinor valued functions in two dimensions. It has been known for some time that Hardy type inequalities yield information about the bound state problem for the three dimensional Dirac-Coulomb equation~\cite{DolEstSer-00}. It is therefore natural to ask whether this relationship continues to hold for the two dimensional Dirac-Coulomb problem and whether it continues to hold in presence of a solenoid. Let us emphasize that we take the Coulomb potential to be $1/r$ and not $-\ln r$. It turns out that a number of spectral properties such as the eigenvalues and eigenfunctions can be worked out in an elementary fashion. Another point of interest is that the approach to self-adjoint extensions through Hardy inequalities as pioneered in~\cite{Esteban-Loss-1, Esteban-Loss-2} works also for this case. In addition to working out the \emph{ground state energy} which falls into the gap, we also get a \emph{critical magnetic field} for which the ground state energy hits the value $0$ and beyond which the operator cannot be further defined, on the basis of pure energy considerations, as a self-adjoint operator. As the field strength approaches this critical value, the slope of the eigenvalue as a function of the field strength tends to negative infinity. We also exhibit the corresponding eigenfunction for all magnetic field strengths below the critical value.

An additional bonus is that all the quantities are given explicitly which allows to investigate in a simple fashion the non-relativistic limit. The presence of the magnetic field in the Dirac equation manifests itself entirely through the vector potential. Formally the Pauli equation should be related with as a non-relativistic limit and involve the magnetic field. The problem, as mentioned above, is that the magnetic field is a delta function at the origin, \emph{i.e.}, it appears as a point interaction. Such situations were investigated before in~\cite{Erdoes-Vougalter} where self-adjoint extensions for the Pauli Hamiltonian involving magnetic point interactions are constructed. 

\medskip Let $\boldsymbol\sigma=(\sigma_i)_{i=1,2,3}$ be the Pauli-matrices defined by
\[
\sigma_1:=\pmatrix{0&1\\1&0}\,,\quad\sigma_2:=\pmatrix{0&-i\\i&0}\,,\quad\sigma_3:=\pmatrix{1&0\\0&-1}\,.
\]
Throughout this paper, let us consider on $\R^2$ the Aharonov-Bohm magnetic potential
\[
\A(x):=\frac a{\rho^2}\,\pmatrix{-x_2\\x_1}\quad\mbox{where}\quad\rho=|x|=\sqrt{x_1^2+x_2^2}\,.
\]
Here $a$ is a real parameter and $(x_1,x_2)$ are standard Cartesian coordinates of $x\in\R^2$. The magnetic gradient is defined as
\[
\nA:\\=\nabla-i\,\A\,.
\]
In atomic units ($m=\hbar=c=1$) the $2d$ magnetic Dirac operator can be written as
\[
H_a:=\sigma_3-i\,\boldsymbol{\sigma}\cdot\nA=\pmatrix{1&\D\\
\D^*&-1}
\]
where $z=x_1+i\,x_2$ and $\bar z=x_1-i\,x_2$. Here
\[
\D:=-\,2\,i\,\dz+i\,a\,\frac{\bar z}{|z|^2}\quad\mbox{and}\quad\D^*:=-\,2\,i\,\dzbar-i\,a\,\frac z{|z|^2}\,,
\]
where
\[
\dz:=\frac12\,(\partial_1-i\,\partial_2)\quad\mbox{and}\quad\dzbar:=\frac12\,(\partial_1+i\,\partial_2)
\]
are the \emph{Wirtinger derivatives}. In polar coordinates such that $\rho=|x|=|z|$ and $\theta=\arctan\(x_2/x_1\)$ so that $z=\rho\,e^{i\theta}$, we have
\[
2\,\dz=e^{-i\theta}\(\partial_\rho-\tfrac i\rho\,\partial_\theta\),\quad2\,\dzbar=e^{i\theta}\(\partial_\rho+\tfrac i\rho\,\partial_\theta\)
\]
and
\[
\D=-\,i\,e^{-i\theta}\(\partial_\rho-\tfrac i\rho\,\partial_\theta-\tfrac a\rho\)\quad\mbox{and}\quad\D^*=-\,i\,e^{i\theta}\(\partial_\rho+\tfrac i\rho\,\partial_\theta+\tfrac a\rho\),
\]
For self-adjointness properties of $H_a$, we refer to~\cite{Sitenko_2000, Falomir_2001, Cacciafesta_Fanelli_2017}.

\medskip For any $\nu\in(0,1/2]$, let us consider the \emph{magnetic Dirac-Coulomb operator}
\[
H_{\nu,a}:=H_a-\frac\nu{|x|}\,.
\]
Our purpose is to establish the range of $a$ for which $H_{\nu,a}$ is self-adjoint for a well chosen domain. According to the approach in~\cite{Esteban-Loss-1}, the self-adjointness is based on the inequality
\be{Hardy}
\int_{\R^2}\(\frac{|\D^*\varphi|^2}{1+\lambda+\frac\nu{|x|}}+\Big(1-\lambda-\tfrac\nu{|x|}\Big)\,|\varphi|^2\)\,dx\ge0\quad\forall\,\varphi\in C^\infty_{\rm c}\big(\R^2\setminus\{0\},\C\big)
\ee
for some $\lambda\in(-1,1)$. For every $\nu\in(0,1/2]$, let us define the \emph{critical magnetic field}~by
\[
\acrit(\nu):=\sup\Big\{\mathsf a>0\,:\,\exists\,\lambda\in(-1,1)\,\mbox{such that~\eqref{Hardy} holds true for any}\,a\in[0,\mathsf a)\Big\}\,.
\]
Our first result deals only with~\eqref{Hardy}. Let us define the function
\[
\crit(s):=\frac12-s\,.
\]
\begin{thm}[Hardy inequality for the magnetic Dirac-Coulomb operator]\label{Thm:Hardy} For any $\nu\in(0,1/2)$, we have
\[
\acrit(\nu)=\crit(\nu)
\]
and~\eqref{Hardy} holds true for any $a\in(0,\acrit(\nu)]$ and $\lambda\in(-1,\lambda_{\nu,a}]$ where
\be{lambda-nu-a}
\lambda_{\nu,a}:=\frac{\sqrt{\crit(a)^2-\nu^2}}{\crit(a)}\,.
\ee
\end{thm}
Notice that $a\le\crit(\nu)$ is equivalent to $\nu\le\crit(a)$. Under this condition, the quadratic form
\be{quadratic}
\varphi\mapsto\mathcal Q_{\nu,a,\lambda}(\varphi):=\int_{\R^2}\(\frac{|\D^*\varphi|^2}{1+\lambda+\frac\nu{|x|}}+\(1-\lambda-\tfrac\nu{|x|}\)\,|\varphi|^2\)\,dx
\ee
is nonnegative on $C^\infty_{\rm c}\big(\R^2\setminus\{0\},\C\big)$ for any $\lambda\in(-1,\lambda_{\nu,a}]$. Since $\mathcal Q_{\nu,a,\lambda}$ is associated with a symmetric operator, it is closable. The results of~\cite{Esteban-Loss-1} can be adapted as follows.
\begin{thm}[Self-adjointness of the magnetic Dirac-Coulomb operator]\label{Thm:critical-a-HB} Let $\nu\in(0,1/2]$. For any $a\in(0,\acrit(\nu)]$, the quadratic form $\mathcal Q_{\nu,a,\lambda}$ is closable and its form domain is a Hilbert space $\mathcal F_{\nu,a}$ which does not depend on $\lambda\in (-1, 1)$. Let
\[
\mathcal D_{\nu,a}:=\Big\{\psi=(\varphi,\chi)^\top\in L^2(\R^2,\C^2)\,:\,\varphi\in\mathcal F_{\nu,a}\,,\;H_{\nu,a}\,\psi\in L^2(\R^2,\C^2)\Big\}
\]
where $ H_{\nu,a}\,\psi$ is understood in the sense of distributions. Then the operator $H_{\nu,a}$ with domain~$\mathcal D_{\nu,a}$ is self-adjoint. Additionally, if $a<\acrit(\nu)$, then $\mathcal F_{\nu,a}$ does not depend on $\nu$ and
\be{subcriticalF+}
\mathcal F_{\nu,a}=\left\{\varphi\in L^2(\R^2,\C)\cap H^1_{\rm loc}\big(\R^2\setminus\{0\},\C\big)\,:\,\sqrt{\tfrac{|x|}{1+|x|}}\,\D^*\varphi\in L^2(\R^2,\C)\right\}\,.
\ee
\end{thm}
In other words, the domain~$\mathcal D_{\nu,a}$ of $H_{\nu,a}$ is the space of spinors
\begin{multline*}
\psi=\pmatrix{\varphi\\\chi}=(\varphi,\chi)^\top\quad\mbox{with}\quad\varphi\in\mathcal F_{\nu,a}\quad\mbox{and}\quad\chi\in L^2(\R^2,\C)\quad\mbox{such that}\\
\D^*\varphi-\(1+\tfrac\nu{|x|}\)\chi\in L^2(\R^2,\C)\quad\mbox{and}\quad\D\chi+\(1-\tfrac\nu{|x|}\)\varphi\in L^2(\R^2,\C)\,.
\end{multline*}
Here $\D^*\varphi$, $\varphi/|x|$, $\D\chi$, and $\chi/|x|$ are interpreted in the sense of distributions. The operator $H_{\nu,a}$ acts on the two components of the spinor $\psi$ and we shall say that $\varphi$ is the \emph{upper component} and $\chi$ the \emph{lower component}. The claim of Theorem~\ref{Thm:critical-a-HB} is that $H_{\nu,a}$ with domain $\mathcal D_{\nu,a}$ is self-adjoint if the field $a$ is at most equal to the critical magnetic field $\acrit(\nu)$. This critical field manifests itself yet in another way. We recall (see for instance~\cite[Theorem~4.7]{Thaller-1992}) that the essential spectrum of $H_{\nu,a}$ is $(-\infty,-1]\cup[1,+\infty)$. Even if the operator $H_{\nu,a}$ is not bounded from below, there is a notion of \emph{ground state}, which also makes sense in the non-relativistic limit (see Section~\ref{Sec:NRL}).
\begin{thm}[Ground state energy of the magnetic Dirac operator]\label{Thm:ground-state} Let $\nu\in(0,1/2]$ and $a\in(0,\acrit(\nu)]$. Then $\lambda_{\nu,a}$ is the lowest eigenvalue in $(-1,1)$ of the operator $H_{\nu,a}$ with domain $\mathcal D_{\nu,a}$.\end{thm}
In the subcritical and critical range, the ground state energy $\lambda_{\nu,a}$ of the magnetic Dirac-Coulomb operator is given by~\eqref{lambda-nu-a} and the ground state itself can be computed: see Proposition~\ref{Prop:qualitative}. The Hardy inequality~\eqref{Hardy} is our key tool in the analysis of the magnetic Dirac-Coulomb operator. This deserves an explanation. If $\psi\in\mathcal D_{\nu,a}$ is an eigenfunction of $H_{\nu,a}$ associated with an eigenvalue $\lambda\in(-1,1)$, then $H_{\nu,a}\,\psi=\lambda\,\psi$ can be rewritten as a system for the upper and lower components $\varphi$ and $\chi$, namely
\be{EL}
\D\chi+\(1-\lambda-\frac\nu{|x|}\)\varphi=0\,,\quad\D^*\varphi-\(1+\lambda+\frac\nu{|x|}\)\chi=0\,.
\ee
By eliminating $\chi$ from the second equation, we find that $\varphi$ is a critical point of the quadratic form $\mathcal Q_{\nu,a,\lambda}$ defined by~\eqref{quadratic} which moreover realizes the equality case in~\eqref{Hardy}. We aim at characterizing $\lambda_{\nu,a}$ as the minimum of a variational problem, which further justifies why we call it a \emph{ground state} energy.

\medskip Our strategy is to prove~\eqref{Hardy} directly using the Aharonov-Casher transformation
\be{psi-eta}
\psi(x)=\pmatrix{|x|^{-a}\,\eta_1(x)\\|x|^a\,\eta_2(x)}\quad\forall\,x\in\R^2
\ee
with $\eta_i:\R^2\to\C$, for $i=1$, $2$. System~\eqref{EL} amounts to
\be{psi-eta-bis}\begin{array}{c}
\eta_1-2\,i\,\rho^{2a}\,\dz\eta_2-\tfrac\nu\rho\,\eta_1=\lambda\,\eta_1\,,\\[8pt]
-2\,i\,\rho^{-2a}\,\dzbar\eta_1-\tfrac\nu\rho\,\eta_2-\eta_2=\lambda\,\eta_2\,.
\end{array}\ee
As for~\eqref{EL}, using the second equation, we can eliminate the lower component and obtain
\be{eta2}
\eta_2=-\,2\,i\,\frac{\rho^{-2a}\,\dzbar\eta_1}{1+\lambda+\tfrac\nu\rho}\,.
\ee
On the space
\[
\mathcal G_{\nu,a}:=\left\{\eta\in L^2\(\R^2,\C; |x|^{-2a}\,dx\)\,:\,|x|^{-a}\eta(x)\in\mathcal F_{\nu,a}\right\}\,,
\]
let us define the counterpart of $\mathcal Q_{\nu,a,\lambda}$ as in~\eqref{quadratic}, that is,
\[\label{J}
J(\eta,\nu,a,\mu):=\int_{\R^2}\(\frac{4\,|\dzbar\eta|^2}{1+\mu+\frac\nu{|x|}}+\Big(1-\mu-\tfrac\nu{|x|}\Big)\,|\eta|^2\)\frac{dx}{|x|^{2a}}\,.
\]
The map $\eta\mapsto J(\eta,\nu,a,\lambda)$ is differentiable and we read from~\eqref{psi-eta-bis} that $\eta_1$ is a critical point after eliminating $\eta_2$ using~\eqref{eta2}.

For a given function $\eta\in\mathcal G_{\nu,a}$, if $J(\eta,\nu,a,0)$ is finite, then $\mu\mapsto J(\eta,\nu,a,\mu)$ is well defined and monotone nonincreasing on $(-1,+\infty)$, with $\lim_{\mu\to+\infty}J(\eta,\nu,a,\mu)=-\,\infty$. As a consequence, the equation $J(\eta,\nu,a,\mu)=0$ has one and only one solution in $(-1,+\infty)$ if, for instance, $\lim_{\mu\to(-1)_+}J(\eta,\nu,a,\mu)>0$. Let us denote this solution by $\lambda_\star(\eta,\nu,a)$, so that
\[\label{lambdastar}
J\big(\eta,\nu,a,\lambda_\star(\eta,\nu,a)\big)=0\,.
\]
By convention, we take $\lambda_\star(\eta,\nu,a)=+\infty$ if the equation has no solution in $(-1,+\infty)$. The main technical estimate and the key result for proving Theorems~\ref{Thm:Hardy},~\ref{Thm:critical-a-HB} and~\ref{Thm:ground-state} goes as follows. 
\begin{prop}[Variational characterization]\label{Prop:critical-a-HB2} For any $\nu\in(0,1/2]$ and $a\in[0,\crit(\nu)]$, 
\be{minim}
\min_{\eta\in\mathcal G_{\nu,a}}\lambda_\star(\eta,\nu,a)=\lambda_{\nu,a}
\ee
where $\lambda_{\nu,a}$ is defined by~\eqref{lambda-nu-a}, and the equality case is achieved, up to multiplication by a constant, by
\[
\eta_\star(x)=|x|^{\sqrt{\crit(a)^2-\nu^2}-\crit(a)}\,e^{-\,\frac\nu{\crit(a)}\,|x|}\quad\forall\,x\in\R^2\setminus\{0\}\,.
\]
\end{prop}
Proposition~\ref{Prop:critical-a-HB2} is at the core of the paper. Let us notice that $\lambda_{\nu,a}$ is the largest $\lambda>-1$ such that $J(\eta,\nu,a,\lambda)\ge0$ for all $\eta\in\mathcal G_{\nu,a}$. This is a \emph{Hardy-type inequality} which deserves some additional considerations. A simple consequence of Proposition~\ref{Prop:critical-a-HB2} is indeed the fact that
\be{IneqJ}
J(\eta,\nu,a,\mu)\ge0\quad\forall\,\eta\in C^\infty_{\rm c}\big(\R^2\setminus\{0\},\C\big)
\ee
for any $\mu\in(-1,\lambda_{\nu,a}]$. Inequality~\eqref{IneqJ} provides us with the interesting inequality
\be{reducedDirac-Hardy}
\int_{\R^2}\(\frac{4\,|\dzbar\eta|^2}{1+\lambda_{\nu,a}+\frac\nu{|x|}}+\Big(1-\lambda_{\nu,a}-\tfrac\nu{|x|}\Big)\,|\eta|^2\)\frac{dx}{|x|^{2a}}\ge0\quad\forall\,\eta\in C^\infty_{\rm c}\big(\R^2\setminus\{0\},\C\big)
\ee
in the case $\mu=\lambda_{\nu,a}$. Using~\eqref{psi-eta}, we already prove~\eqref{Hardy} written for $\lambda=\lambda_{\nu,a}$, namely
\be{HardyOpt}
\int_{\R^2}\(\frac{|\D^*\varphi|^2}{1+\lambda_{\nu,a}+\frac\nu{|x|}}+\Big(1-\lambda_{\nu,a}-\tfrac\nu{|x|}\Big)\,|\varphi|^2\)\,dx\ge0\quad\forall\,\varphi\in C^\infty_{\rm c}\big(\R^2\setminus\{0\},\C\big)\,.
\ee

It is an essential property of the Aharonov-Bohm magnetic field that~\eqref{psi-eta} transforms the quadratic form associated with~\eqref{HardyOpt} into~\eqref{reducedDirac-Hardy}, which is a weighted inequality, without magnetic field. In terms of scalings, one has to think of a Hardy type inequality for a Dirac-Coulomb operator in a non-integer dimension $2-2\,a$. By taking the limit of the inequality $J(\eta,\nu,a,\mu)\ge0$ as $\mu\to(-1)_+$, we obtain
\[
\int_{\R^2}\(\frac4\nu\,|x|^{1-2a}\,|\dzbar\eta|^2+2\,\frac{|\eta|^2}{|x|^{2a}}\)dx\ge\nu\int_{\R^2}\frac{|\eta|^2}{|x|^{2a+1}}\,dx\quad\forall\,\eta\in C^\infty_{\rm c}\big(\R^2\setminus\{0\},\C\big)
\]
under the assumption $a\in(0,1/2)$ and $\nu\in(0,\crit(a))$. Using scalings, we can get rid of the non-homogeneous term and find by taking the limit as $\nu\to\crit(a)$ that
\[
\int_{\R^2}|x|^{1-2a}\,|\dzbar\eta|^2\,dx\ge\frac14\,\crit(a)^2\int_{\R^2}\frac{|\eta|^2}{|x|^{2a+1}}\,dx\quad\forall\,\eta\in C^\infty_{\rm c}\big(\R^2\setminus\{0\},\C\big)\,.
\]
These weighted magnetic Hardy inequalities are part of a larger family of inequalities.
\begin{lem}[Hardy inequalities for Wirtinger derivatives]\label{Lem:partialHardyineq} Let $\beta\in\R$. Then we have the two following inequalities (with optimal constants)
\be{Wirtinger1}
\int_{\R^2}|x|^{-2\beta}\,|\dzbar\eta|^2\,dx\ge\frac14\,\min_{\ell\in\Z}\,\(\beta-\ell\)^2\int_{\R^2}|x|^{-2\beta-2}\,|\eta|^2\,dx\quad\forall\,\eta\in C^\infty_{\rm c}\big(\R^2\setminus\{0\},\C\big)\,,
\ee
\be{Wirtinger2}
\int_{\R^2}|x|^{2\beta}\,|\dz\eta|^2\,dx\ge\frac14\,\min_{\ell\in\Z}\,\(\beta-\ell\)^2\int_{\R^2}|x|^{2\beta-2}\,|\eta|^2\,dx\quad\forall\,\eta\in C^\infty_{\rm c}\big(\R^2\setminus\{0\},\C\big)\,.
\ee\end{lem}
As a simple consequence of Lemma~\ref{Lem:partialHardyineq}, we also obtain a family of Hardy inequalities for spinors corresponding to the \emph{magnetic Pauli operator} $-i\,\boldsymbol{\sigma}\cdot\nA$ (see~~\cite{Erdoes-Vougalter}).
\begin{cor}[Hardy inequalities for the magnetic Pauli operator]\label{Cor:ABPauli} Let $\zeta\in\R$, $a\in[0, 1/2]$. Then
\be{Magn-Hardy-Pauli-intro}
\int_{\R^2}|x|^\zeta\,|(\boldsymbol{\sigma}\cdot\nA)\,\psi|^2\,dx\ge C_{a,\zeta}\int_{\R^2}|x|^{\zeta-2}\,|\psi|^2\,dx\quad\forall\,\psi\in C^\infty_{\rm c}(\R^2\setminus\{0\},\C^2)
\ee
and the optimal constant in inequality~\eqref{Magn-Hardy-Pauli-intro} is given by
\[
C_{a,\zeta}=\min_{\pm\,,\,\ell\in\Z}\(a\pm\tfrac\zeta2-\ell\)^2\,.
\]
\end{cor}
A straightforward consequence of the expression of $C_{a,\zeta}$ is that, for all $\zeta\in\R$, the function $a\mapsto C_{a,\zeta}$ is $1$-periodic. More inequalities of interest are listed in Appendices~\ref{Sec:SpecialHardy} and~\ref{Sec:Positron}.

\medskip Let us give an overview of the literature. In absence of magnetic field, the computation of the eigenvalues of the hydrogen atom in the setting of the Dirac equation goes back to~\cite{1928,gordon1928energieniveaus}. As noted in~\cite{Mawhin_2010}, an explicit computation of the spectrum can be done using Laguerre polynomials. By the transformation~\eqref{psi-eta}, which replaces a problem with magnetic field by an equivalent problem with weights, we obtain a similar algebra, which would allow us to adapt the computations of~\cite{Dong-Ma_2003}. As we are interested only in the \emph{ground state energy}, we use a simpler approach based on Hardy-type inequalities for the upper component of the Dirac operator: the inequality follows from a simple \emph{completion of a square} as in~\cite[Remark,~p.~9]{MR2091354}. Notice that~\eqref{psi-eta} can be seen as a special case of the transformation introduced by Aharonov and Casher in~\cite{PhysRevA.19.2461} and used to define the Pauli operator for some measure valued magnetic fields (see~\cite[Inequality~(3)]{Erdoes-Vougalter}). Although rather elementary, the computation in dimension two of the ground state and the ground state energy for the magnetic Dirac-Coulomb operator in presence of the Aharonov-Bohm magnetic field is, as far as we know, original.

As noted in~\cite{Esteban-Lewin-Sere-2019}, the \emph{two-dimensional case} is relevant in solid state physics for the study of graphene: the low-energy electronic excitations are modeled by a massless two-dimensional Dirac equation~\cite{RevModPhys.81.109} but the study of strained graphene involves a massive Dirac operator~\cite{Vozmediano_2010}. Magnetic impurities are known to play a role. The coupling with Aharonov-Bohm magnetic fields raises interesting mathematical questions which have probably interesting consequences from the experimental point of view. 

Our interest for \emph{Hardy-type inequalities} in the presence of an Aharonov-Bohm magnetic field goes back to the inequality proved by Laptev and Weidl in~\cite{Laptev-Weidl}. In the perspective of Schr\"odinger operators, such an inequality in dimension two is somewhat surprising, because it is well known that there is no such inequality without magnetic field. It is therefore natural to investigate whether there is a relativistic counterpart, which is our purpose in this paper, as well as to consider the non-relativistic limit. The link between ground states for Dirac operators and Hardy inequalities is known for instance from~\cite[page~222]{DolEstSer-00} and has been exploited in works like~\cite{MR2091354,DDEV,MR2379440}. Here we have a new example which is particularly interesting as the ground state is explicit and its upper component realizes the equality case in the corresponding Hardy inequality.

The continuous spectrum of $H_{\nu,a}$ depends neither on $\nu$ nor on $a$: see~\cite[Theorem~4.7]{Thaller-1992} and~\cite{Hogreve_2012}. If $a=0$, see~\cite{MR1870409} for all $\nu>0$ and also~\cite{Klaus_1979b,MR1451618} and~\cite{Arai_1982,Arai_1983,Nenciu_1976}. The case $a\neq0$ is less standard but does not raise additional difficulties. It is well known that the lower part of the continuous spectrum prevents $H_{\nu,a}$ to be semi-bounded from below in any reasonable sense. In order to characterize the eigenvalues in the gap, one has to address more subtle \emph{variational principles} than standard Rayleigh quotients. The first min-max formulae based on a decomposition into an upper and a lower component of the free Dirac operator perturbed by an electrostatic potential with Coulomb-type singularity at the origin were proposed by Talman in~\cite{PhysRevLett.57.1091} and Datta-Devaiah in~\cite{Datta_1988}. From the mathematical point of view, \emph{min-max methods} for the characterization of eigenvalues go back to~\cite{MR1479240,MR1724845,DolEstSer-00,MR1767717}. See also~\cite{Dolbeault-Esteban-Loss-06,MorMul-15,Mueller-2016} for more recent results and especially~\cite{SchSolTok-20} which provides a comprehensive overview. Such a variational approach provides numerical schemes for computing Dirac eigenvalues, which have been studied in~\cite{Lenthe_1993,PhysRevLett.85.4020,dolbeault2003variational,Kullie_2004,Zhang_2004,PhysRevA.72.022103}. The symmetry of the electron case and the \emph{positron} case is inherent to the Dirac equation: see for instance~\cite[Chapter~11]{Schwabl_2004} for a review of the classical invariances associated with the Dirac operator. This symmetry has been reformulated from a variational point of view in~\cite{MR2239275}. In the present paper, this is exploited in Appendix~\ref{Sec:Positron} for obtaining a dual family of Hardy-type inequalities, which is formally summarized by the transformation $(\nu,a,\lambda)\mapsto(-\nu,-a,-\lambda)$, $\D\mapsto\D^*$ and $\dzbar\mapsto\dz$. Non-relativistic limits are from the early papers~\cite{1928,gordon1928energieniveaus} a natural question and have been studied more recently, for instance, in the context of the Dirac-Fock model in~\cite{Ounaies-2000,MR1869528,MR2602013}.

A delicate issue for Dirac operators with singular Coulomb potentials defined on smooth functions is the determination of a self-adjoint extension. According to~\cite[Theorem~2.1]{Hogreve_2012}, there are three different regimes of the Dirac-Coulomb operator on $\R^3$. If $0<\nu\le\sqrt3/2$, the operator is essentially self-adjoint, with finite potential and kinetic energies. In the interval $\sqrt3/2<\nu<1$, besides other self-adjoint extensions, distinguished self-adjoint extensions can be singled out (see~\cite{Schmincke_1972b,Schmincke_1972,10.1007/BFb0067087}) for which either the potential energy (see~\cite{W_st_1973,W_st_1975,W_st_1977}) or the kinetic energy (see~\cite{Nenciu_1976,MR462346}) are finite. All these extensions were proved to be equivalent by Klaus and W\"ust in~\cite{Klaus_1979} and coincide with the distinguished extension of~\cite{Esteban-Loss-1}. In the critical case $\nu=1$, Hardy-Dirac inequalities lead to a distinguished self-adjoint extension with finite total energy: see~\cite{Esteban-Loss-1,Esteban-Loss-2,Arrizabalaga_2011}. For $\nu>1$, the operator enjoys a family of self-adjoint extensions, but standard finiteness of the energies fail according to~\cite{Burnap_1981,MR1451618,Voronov_2007}.

On $\R^2$, the self-adjointness properties of the operator $H_a$ (without Coulomb potential) have been studied for instance in~\cite{Sitenko_2000, Falomir_2001, Cacciafesta_Fanelli_2017}. Without magnetic field, self-adjoint extensions preserving gaps have been studied in~\cite{MR0024574,MR1261368} and it is shown there that the decomposition used in the variational approach of~\cite[Theorem~1.1]{SchSolTok-20} enters in the framework of Krein's criterion (see ~\cite[Remark~1.3]{SchSolTok-20}). All self-adjoint extensions of the Dirac-Coulomb operator for $\nu<1$ are classified in~\cite{MR3933559} with corresponding eigenvalues obtained in~\cite{MR3817548}. As noted in~~\cite{SchSolTok-20}, methods similar to those of~\cite{Esteban-Loss-2} are applied in~\cite{MR3962850} to a two-body Dirac operator with Coulomb interaction (without spectral gap). The characterization of the domain of the extension of the operator is of course a very natural question. This is studied in detail in~\cite{Esteban-Lewin-Sere-2019}, including the two-dimensional case but without magnetic field (also see references therein). Some of the results in this paper are natural extensions of the considerations in~\cite{Esteban-Lewin-Sere-2019} for the case without magnetic field.

\medskip This paper is organized as follows. Section~\ref{Sec:homog-Hardy} is devoted to a direct proof of the homogeneous Hardy-like inequalities of Lemma~\ref{Lem:partialHardyineq} and Corollary~\ref{Cor:ABPauli}. In Section~\ref{Sec:GS}, we prove some inhomogeneous Hardy-like inequalities which allow us to study a variational problem leading to the proof of Proposition~\ref{Prop:critical-a-HB2} and Theorem~\ref{Thm:Hardy}. This takes us to the identification of the \emph{critical magnetic field} under which all our results hold true. This also proves the Hardy-type inequality~\eqref{reducedDirac-Hardy} and its consequences (also see Corollary~\ref{Cor:Thm1a} and Appendix~\ref{Sec:SpecialHardy}). Section~\ref{Sec:mainthms} is devoted to the self-adjointness properties of $H_{\nu,a}$ and to the identification of the \emph{ground state}. An explicit expression is found, which is explained by the role played by Laguerre polynomials in the computation of the spectrum of $H_{\nu,a}$ (see Appendix~\ref{Sec:Laguerre}). The non-relativistic limit is discussed in Section~\ref{Sec:NRL}. The inequalities corresponding to the positron case are collected in Appendix~\ref{Appendix}.

\section{Homogeneous Hardy-like inequalities}\label{Sec:homog-Hardy}

In this section, we prove the inequalities of Lemma~\ref{Lem:partialHardyineq} and Corollary~\ref{Cor:ABPauli}, which are independent of our other results, but rely on the completion of squares and on~\eqref{psi-eta}. These proofs are a useful introduction to the main results of this paper.

\begin{proof}[Proof of Lemma~\ref{Lem:partialHardyineq}.] Let us start with the proof of~\eqref{Wirtinger2}. For any $\ell\in\Z$, if $\eta(\rho,\theta)=e^{i\,\ell\,\theta}\, \phi(\rho)$, then
\begin{align*}
4\int_{\R^2}|x|^{2\beta}\,|\dz\eta|^2\,dx&=\int_{\R^2}\rho^{2\beta}\,\Big|\(\partial_\rho-i\,\tfrac{\partial_\theta}\rho\) e^{i\,\ell\,\theta}\phi\,\Big|^2\,dx\\
&=2\,\pi\int_0^{+\infty}\rho^{2\beta+1}\,\Big|\(\partial_\rho+\tfrac\ell\rho\)\phi\,\Big|^2\,d\rho\,,\\
\int_{\R^2}|x|^{2\beta-2}\,|\eta|^2\,dx&=2\,\pi\int_0^{+\infty}\rho^{2\beta-1}\,|\phi|^2\,d\rho\,.
\end{align*}
With an expansion of the square and an integration by parts with respect to $\rho$, we obtain
\[
\int_0^{+\infty}\rho^{2\beta+1}\,\Big|\partial_\rho\phi-\tfrac\kappa\rho\,\phi\,\Big|^2\,d\rho=\int_0^{+\infty}\rho^{2\beta+1}\,|\partial_\rho\phi|^2\,d\rho+\kappa\,(\kappa-2\,\beta)\int_0^{+\infty}\rho^{2\beta-1}\,|\phi|^2\,d\rho
\]
for any $\kappa\in\R$. Applied with $\kappa=\beta$ and with $\kappa=-\,\ell$, this proves that
\[
\int_0^{+\infty}\rho^{2\beta+1}\,|\partial_\rho\phi|^2\,d\rho\ge\beta^2\int_0^{+\infty}\rho^{2\beta-1}\,|\phi|^2\,d\rho
\]
and
\begin{multline*}
\int_0^{+\infty}\rho^{2\beta+1}\,\Big|\partial_\rho\phi+\tfrac\ell\rho\,\phi\,\Big|^2\,d\rho\\
=\int_0^{+\infty}\rho^{2\beta+1}\,|\partial_\rho\phi|^2\,d\rho+\ell\,(\ell-2\,\beta)\int_0^{+\infty}\rho^{2\beta-1}\,|\rho\phi|^2\,d\rho\\
\ge(\beta-\ell)^2\int_0^{+\infty}\rho^{2\beta-1}\,|\phi|^2\,d\rho\,.
\end{multline*}
Altogether, we have
\[
\int_{\R^2}|x|^{2\beta}\,|\dz\eta|^2\,dx\ge\frac14\,\(\beta-\ell\)^2\int_{\R^2}|x|^{2\beta-2}\,|\eta|^2\,dx\,.
\]
Inequality~\eqref{Wirtinger2} for a general $\eta\in C^\infty_{\rm c}\big(\R^2\setminus\{0\},\C\big)$ is then a consequence of a decomposition in Fourier modes. The proof of~\eqref{Wirtinger1} is exactly the same, up to the sign changes $\beta\mapsto-\,\beta$ and $\ell\mapsto-\,\ell$. \end{proof}

\begin{proof}[Proof of Corollary~\ref{Cor:ABPauli}.] Using~\eqref{psi-eta}, Inequality~\eqref{Magn-Hardy-Pauli-intro} is equivalent to
\begin{multline*}
4\int_{\R^2}\(|x|^{\zeta-2a}\,|\dzbar\eta_1|^2+|x|^{\zeta+2a}\,|\dz\eta_2|^2\)dx\\
\ge C_{a,\zeta}\int_{\R^2}\(|x|^{\zeta-2-2a}\,|\eta_1|^2+|x|^{\zeta-2+2a}\,|\eta_2|^2\)dx
\end{multline*}
and the result follows by applying~\eqref{Wirtinger1} and~\eqref{Wirtinger2} to $\eta_1$ and $\eta_2$ with respectively $\eta=a-\zeta/2$ and $\eta=a+\zeta/2$.\end{proof}

\section{The minimization problem}\label{Sec:GS}

The goal of this section is to prove Proposition~\ref{Prop:critical-a-HB2} and Theorem~\ref{Thm:Hardy}. It is centred on the variational problem~\eqref{minim} and its consequences on the definition of the \emph{critical magnetic field}.
\begin{lem}\label{Lem:Square} Assume that $\nu\in(0,1/2]$ and $a\in[0,\crit(\nu)]$. Then
\be{mu}
\mu=\frac{\crit(a)-\sqrt{\crit(a)^2-\nu^2}}\nu
\ee
is the smallest value of $\mu\in\R$ such that the inequality
\[
\int_0^{+\infty}\frac{|\phi'|^2}{\nu+\rho}\,\rho^{2-2a}\,d\rho+\mu^2\int_0^{+\infty}|\phi|^2\,\rho^{1-2a}\,d\rho\ge\nu\int_0^{+\infty}|\phi|^2\,\rho^{-2a}\,d\rho
\]
holds for any function $\phi\in C^1((0,+\infty),\C)$. Moreover, the equality case with $\mu$ given by~\eqref{mu} is achieved by
\[
\phi_\star(\rho)=\rho^{-\mu\,\nu}\,e^{-\mu\,\rho}\quad\forall\,\rho\in(0,+\infty)\,.
\]\end{lem}
\begin{proof} An expansion of the square and an integration by parts show that
\begin{align*}
0\le&\int_0^{+\infty}\frac{|\rho\,\phi'+\mu\,(\nu+\rho)\,\phi|^2}{\nu+\rho}\,\rho^{-2a}\,d\rho\\
&=\int_0^{+\infty}\frac{\rho^2\,|\phi'|^2}{\nu+\rho}\,\rho^{-2a}\,d\rho+\mu\int_0^{+\infty}\(|\phi|^2\)'\,\rho^{1-2a}\,d\rho\\
&\hspace*{4cm}+\mu^2\int_0^{+\infty}(\nu+\rho)\,|\phi|^2\,\rho^{-2a}\,d\rho\\
&=\int_0^{+\infty}\frac{|\phi'|^2}{\nu+\rho}\,\rho^{2-2a}\,d\rho+\mu^2\int_0^{+\infty}|\phi|^2\,\rho^{1-2a}\,d\rho\\
&\hspace*{4cm}+\(\nu\,\mu^2-(1-2\,a)\,\mu\)\int_0^{+\infty}|\phi|^2\,\rho^{-2a}\,d\rho\,.
\end{align*}
The conclusion holds after observing that $\nu\,\mu^2-(1-2\,a)\,\mu=-\nu$ and solving, in the equality case, the equation
\[
\rho\,\phi'+\mu\,(\nu+\rho)\,\phi=0\quad\forall\,\rho\in(0,+\infty)\,.
\]
\end{proof}

Next we apply Lemma~\ref{Lem:Square} to the variational problem~\eqref{minim}.
\begin{proof}[Proof of Proposition~\ref{Prop:critical-a-HB2}.] For any function $\phi$ smooth enough and any $\ell\in\Z\setminus\{0\}$, we have
\[
\int_0^{+\infty}\frac{|\rho\,\phi'-\ell\,\phi|^2}{(1+\lambda)\,\rho+\nu}\,\frac{d\rho}{\rho^{2a}}\ge\int_0^{+\infty}\frac{\rho^2\,|\phi'|^2}{(1+\lambda)\,\rho+\nu}\,\frac{d\rho}{\rho^{2a}}
\]
because an integration by parts shows that
\begin{multline*}
\int_0^{+\infty}\frac{\ell^2\,|\phi|^2-2\,\rho\,\ell\,\phi\,\phi'}{(1+\lambda)\,\rho+\nu}\,\frac{d\rho}{\rho^{2a}}\\
=\int_0^{+\infty}\frac{\ell\,(\ell-2\,a)\,(1+\lambda)\,\rho+\ell\,(\ell+1-2\,a)\,\nu}{\big((1+\lambda)\,\rho+\nu\big)^2}\,|\phi|^2\frac{d\rho}{\rho^{2a}}
\end{multline*}
and $\ell\,(\ell-2\,a)\ge0$ and $\ell\,(\ell+1-2\,a)\ge0$. Using polar coordinates and a decomposition in Fourier modes,
\be{Fourier}
\eta(x)=\sum_{\ell\in\Z}\eta_\ell(\rho)\,e^{i\ell\theta}\,,
\ee
we obtain
\be{JFourierDecomposition}
J(\eta,\nu,a,\lambda)=2\pi\sum_{\ell\in\Z}\int_0^{+\infty}\(\frac{|\rho\,\eta_\ell'-\ell\,\eta_\ell|^2}{(1+\lambda)\,\rho+\nu}+\Big((1-\lambda)\,\rho-\nu\Big)\,|\eta_\ell|^2\)\frac{d\rho}{\rho^{2a}}\,.
\ee
As a consequence, in order to minimize $\lambda_\star(\eta,\nu,a)$, it is enough to consider only the $\ell=0$ mode, \emph{i.e.}, minimize on radial functions.

Let us consider the change of variables $\rho\mapsto(1+\lambda)\,\rho$ and write
\[
\eta(\rho)=\phi\big((1+\lambda)\,\rho\big)\,.
\]
We observe that the largest value of $\lambda>-1$ for which $J(\eta,\nu,a,\lambda)\ge0$ for any $\eta\in\mathcal G_{\nu,a}$ is the largest value of $\lambda>-1$ for which the inequality
\be{Jlambda}
\int_0^{+\infty}\frac{|\phi'|^2}{\nu+\rho}\,\rho^{2-2a}\,d\rho+\frac{1-\lambda}{1+\lambda}\int_0^{+\infty}|\phi|^2\,\rho^{1-2a}\,d\rho-\nu\int_0^{+\infty}|\phi|^2\,\rho^{-2a}\,d\rho\ge0
\ee
holds for any $\phi$. With the notation of Lemma~\ref{Lem:Square}, $(1-\lambda)/(1+\lambda)=\mu^2$ if and only if $\lambda=(1-\mu^2)/(1+\mu^2)$, that is, $\lambda=\sqrt{\crit(a)^2-\nu^2}/\crit(a)=\lambda_{\nu,a}$ according to~\eqref{lambda-nu-a}. Equality in~\eqref{Jlambda} is obtained, up to multiplication by a constant, with \hbox{$\eta_\star(\rho)=\phi_\star\big((1+\lambda)\,\rho\big)$}.\end{proof}

The above proof deserves an observation. In the proof of Lemma~\ref{Lem:Square}, we solve $\nu\,\mu^2-(1-2\,a)\,\mu+\nu=0$. It turns out that for any $\nu<\crit(a)$, this equation has two roots, $\mu_\pm=\(-\,\crit(a)\pm\sqrt{\crit(a)^2-\nu^2}\)/\nu$, and the inequality is true for any $\mu\in[\mu_-,\mu_+]$. In the proof of Proposition~\ref{Prop:critical-a-HB2}, we solve $(1-\lambda)/(1+\lambda)=\mu^2$ and we look for the largest value of $\lambda$ for which $J(\eta,\nu,a,\lambda)\ge0$ for any $\eta$, which is the reason why we pick the value of $\lambda$ corresponding to $\mu=\mu_+$. See Appendix~\ref{Sec:Positron} for similar considerations in the positron case.

\medskip Using the Aharonov-Casher transformation
\be{varphi-eta}
\varphi(x)=|x|^{-a}\,\eta(x)\quad\forall\,x\in\R^2
\ee
which transforms a function $\eta\in\mathcal G_{\nu,a}$ into a function $\varphi\in\mathcal F_{\nu,a}$, we can rephrase the result of Proposition~\ref{Prop:critical-a-HB2} as follows.
\begin{cor}\label{Cor:Thm1a} Under the conditions $\nu\in(0,1/2]$ and $a\in[0,\crit(\nu)]$, the largest value of $\lambda>-1$ for which $\mathcal Q_{\nu,a,\lambda}$ as defined in~\eqref{quadratic} is a nonnegative quadratic form on $\mathcal F_{\nu,a}$ is $\lambda_{\nu,a}$. Additionally, if $a<\crit(\nu)$, the equality case in $\mathcal Q_{\nu,a,\lambda_{\nu,a}}(\varphi)\ge0$ is achieved if and only if $\varphi=\varphi_\star$ up to a multiplicative constant, where
\be{varphistar}
\varphi_\star(\rho,\theta)=\rho^{\kern-2pt\sqrt{\crit(a)^2-\nu^2}-\frac12}\,e^{-\frac\nu{\crit(a)}\,\rho}\quad\forall\,(\rho,\theta)\in\R^+\times[0,2\pi)\,.
\ee
\end{cor}
However, we read from~\eqref{lambda-nu-a} that $\lim_{a\to\crit(\nu)_-}\lambda_{\nu,a}=0>-1$. As we shall see next, as soon as $a\in(\crit(\nu),1/2)$ for some $\nu\in(0,1/2]$, there is no $\lambda\in(-1,1)$ such that the Hardy-like inequality~\eqref{Hardy} holds true anymore.
\begin{prop}\label{Prop:supercritical} Let $\nu\in(0,1/2]$ and $a\in(\crit(\nu),1/2)$. Then
\[
\inf_{\eta\in\mathcal G_{\nu,a}}\lambda_\star(\eta,\nu,a)\le-1
\]
and for any $\lambda\in(-1,1)$, there is some $\varphi\in\varphi\in L^2(\R^2,\C)\cap H^1_{\rm loc}\big(\R^2\setminus\{0\},\C\big)$ such that $\sqrt{|x|/(1+|x|)}\,\D^*\varphi\in L^2(\R^2,\C)$ and $\mathcal Q_{\nu,a,\lambda}(\varphi)<0$.
\end{prop}
\begin{proof} Let $\epsilon>0$, $\mu\in[-1,1)$ and, for any $\rho>0$ and consider a function $\eta_\epsilon\in\mathcal G_{\nu,a}$ such that
\[
\eta_\epsilon(x)=|x|^{a-\frac12}\,e^{-|x|}\quad\mbox{if}\quad|x|\ge\epsilon\quad\mbox{and}\quad|\nabla\eta_\epsilon(x)|\le\epsilon^{a-\frac32}\quad\mbox{if}\quad|x|\le\epsilon\,.
\]
A computation shows the existence of two positive constants $C_1$ and $C_2$ such that
\[
J(\eta_\epsilon,\nu,a,\mu)\le J(\eta_\epsilon,\nu,a,-1)\le C_1\,|\log\epsilon|\(\frac{\crit(a)^2}\nu-\nu\)+C_2<0
\]
for $\epsilon$ small enough, so that $\lambda_\star(\eta_\epsilon,\nu,a)\le-1$. Using the transformation~\eqref{varphi-eta}, this also proves that $\mathcal Q_{\nu,a,\lambda}$ achieves a negative value.
\end{proof}

\begin{proof}[Proof of Theorem~\ref{Thm:Hardy}.] We know from Corollary~\ref{Cor:Thm1a} that $\acrit(\nu)\ge\crit(\nu)$ and from Proposition~\ref{Prop:supercritical} that $\acrit(\nu)\le\crit(\nu)$. Since $\lambda\mapsto\mathcal Q_{\nu,a,\lambda}(\varphi)$ as defined in~\eqref{quadratic} is monotone nonincreasing,~\eqref{HardyOpt} holds true for any $\lambda\in[-1,\lambda_{\nu,a})]$.
\end{proof}

\section{The 2d magnetic Dirac-Coulomb operator with an Aharonov-Bohm magnetic field}\label{Sec:mainthms}

This section is devoted to the self-adjointness of the operator $H_{\nu,a}$ with domain $\mathcal D_{\nu,a}$, when $\nu\in(0,1/2]$ and $a\in(0,\crit(\nu)]$, that is, Theorem~\ref{Thm:critical-a-HB}. In the same range of parameters, we also identify $\lambda_{\nu,a}$ as the \emph{ground-state} of $H_{\nu,a}$ (Theorem~\ref{Thm:ground-state}).

\begin{proof}[Proof of Theorem~\ref{Thm:critical-a-HB}.] We first deal with abstract results when $a\le\acrit(\nu)$ before characterizing $\mathcal F_{\nu,a}$ in the subcritical range $a<\acrit(\nu)$.

\smallskip\noindent$\rhd$ \emph{Critical and subcritical cases, $a\le\acrit(\nu)$: domain and self-adjointness}.
We follow the method of~\cite{Esteban-Loss-2} for dealing with non-magnetic $3d$ Dirac-Coulomb operators in~\cite{Esteban-Loss-1,SchSolTok-20}. This method applies almost without change and we provide only a sketch of the proof. 

By Theorem~\ref{Thm:Hardy}, the quadratic form $\mathcal Q_{\nu,a,\lambda}$ defined by~\eqref{quadratic} is nonnegative on $C^\infty_{\rm c}\big(\R^2\setminus\{0\},\C\big)$ for any $\lambda\in(-1,\lambda_{\nu,a}]$. Let us define the norms $\|\cdot\|$ and $\|\cdot\|_\lambda$ by
\begin{multline*}
\|\varphi\|^2:=\|\varphi\|_{L^2(\R^2,\C)}^2+\left\|\sqrt{\tfrac{|x|}{1+|x|}}\,\D^*\varphi\right\|_{L^2(\R^2,\C)}^2\\
\mbox{and}\quad\|\varphi\|^2_\lambda:=\|\varphi\|_{L^2(\R^2,\C)}^2+\left\|\tfrac{\D^*\varphi}{\sqrt{1+\lambda+\frac\nu{|x|}}}\right\|_{L^2(\R^2,\C)}^2\,.
\end{multline*}
For the same reasons as in~\cite[Lemma~2.1]{DolEstSer-00} or~\cite[Lemma~5]{SchSolTok-20}, the norms $\|\cdot\|$ and $\|\cdot\|_\lambda$ are equivalent for any $\lambda\in(-1,1)$ and the operator $\varphi\mapsto\D^*\varphi/\big(1+\lambda+\nu\,|x|^{-1}\big)$ on $C^\infty_{\rm c}\big(\R^2\setminus\{0\},\C\big)$ is closable with respect to $\|\cdot\|_\lambda$, with a domain that does not depend on $\lambda$. By arguing as in~\cite[Lemma~2.1]{DolEstSer-00} and~\cite[Lemma~7]{SchSolTok-20}, there is a constant $\kappa_\lambda\ge1$ such that $\mathcal Q_{\nu,a,\lambda}(\varphi)+\kappa_\lambda\,\|\varphi\|_\lambda^2$ is equivalent to $\mathcal Q_{\nu,a,0}(\varphi)+\|\varphi\|_0^2$ for all $\lambda\in(-1,1)$. These quadratic forms are nonnegative and closable with form domain $\mathcal F_{\nu,a}\subset L^2(\R^2,\C)$ as defined in Theorem~\ref{Thm:critical-a-HB}. Moreover $\mathcal F_{\nu,a}$ does not depend on $\lambda$ because of the equivalence of the quadratic forms. On $\mathcal D_{\nu,a}$, the operator $H_{\nu,a}$ such that
\[
H_{\nu,a}\,\psi=\pmatrix{\D\chi+\big(1-\frac\nu{|x|}\big)\,\varphi\\[4pt]\D^*\varphi-\big(1+\frac\nu{|x|}\big)\,\chi}\quad\forall\,\psi=\pmatrix{\varphi\\\chi}\in\mathcal D_{\nu,a}
\]
is self-adjoint, for the same reasons as in~\cite{Esteban-Loss-1}, because for all $\lambda\in (-1,\lambda_{\nu,a})$, the operator $H_{\nu,a}-\lambda$ is symmetric and it is a bijection from $\mathcal D_{\nu,a}$ onto $L^2(\R^2,\C^2)$.

\smallskip\noindent$\rhd$ \emph{In the subcritical range $a<\acrit(\nu)$, the space $\mathcal F_{\nu,a}$} endowed with the norm $\|\cdot\|$ is given by~\eqref{subcriticalF+}. Let us prove it. With $C=\min\{1+\lambda,\nu\}$, we have that
\[
\mathcal Q_{\nu,a,\lambda}(\varphi)\le\tfrac1C\int_{\R^2}\tfrac{|x|}{1+|x|}\,|\D^*\varphi|^2\,dx+2\int_{\R^2}|\varphi|^2\,dx\le\max\left\{C^{-1},2\right\}\,\|\varphi\|^2
\]
for any $\varphi\in C^\infty_{\rm c}\big(\R^2\setminus\{0\},\C\big)$. The reverse inequality goes as follows. For any $\lambda\in(-1,\lambda_{\nu,a})$, let
\[
t=\frac{\crit(a)}\nu\,\sqrt{1-\lambda^2}
\]
and notice that $t>1$. This choice of $t$ is made such that $\lambda=\lambda_{t\nu,a}$. Let us choose some $\varphi\in C^\infty_{\rm c}\big(\R^2\setminus\{0\},\C\big)$ and consider $\phi(x)=\varphi\big(t^{-1}x\big)$. A simple change of variables shows that
\begin{multline*}
\int_{\R^2}\(\frac 1{t^2}\,\frac{|\D^*\varphi|^2}{1+\lambda+\frac\nu{|x|}}-\frac\nu{|x|}\,|\varphi|^2\)dx=\frac 1{t^2}\int_{\R^2}\(\frac{|\D^*\phi|^2}{1+\lambda+\frac{t\kern1pt\nu}{|x|}}-\frac{t\nu}{|x|}\,|\phi|^2\)dx\\
\ge\frac{\lambda-1}{t^2}\int_{\R^2}|\phi|^2\,dx=(\lambda-1)\int_{\R^2}|\varphi|^2\,dx
\end{multline*}
where the inequality is obtained by writing $\mathcal Q_{t\nu,a,\lambda}(\phi)\ge0$. As a consequence, we have
\[
\mathcal Q_{\nu,a,\lambda}(\varphi)\ge\(1-\tfrac1{t^2}\)\frac1{\max\{1+\lambda,\nu\}}\int_{\R^2}\tfrac{|x|}{1+|x|}\,|\D^*\varphi|^2\,dx\,,
\]
which concludes the proof.
\end{proof}

\begin{proof}[Proof of Theorem~\ref{Thm:ground-state}.] As noted in the introduction, if $\lambda\in(-1,1)$ is an eigenvalue of $H_{\nu,a}$, its upper component is, after the the Aharonov-Casher transformation~\eqref{varphi-eta}, a critical point of $\eta\mapsto J(\eta,\nu,a,\lambda)$, so that the ground state energy, \emph{i.e.}, the lowest eigenvalue of $H_{\nu,a}$ in $(-1,1)$, is larger or equal than $\lambda_{\nu,a}$. We have equality if $\eta_1=\eta_\star$ determines an eigenfunction of $H_{\nu,a}$ through~\eqref{psi-eta} and~\eqref{eta2}. This is straightforward in the subcritical range as $\eta_\star\in\mathcal G_{\nu,a}$ if $a<\acrit(\nu)$ by Theorem~\ref{Thm:critical-a-HB}. 

The critical case $a=\acrit(\nu)$ is more subtle as we have no explicit characterization of $\mathcal F_{\nu,a}$ or, equivalently,~$\mathcal G_{\nu,a}$. We have indeed to prove that $\eta_\star\in\mathcal G_{\nu,a}$ if $a=\acrit(\nu)$. For all $\epsilon\in(0,1]$, let us define the truncation function
\[
\theta_\epsilon := \begin{cases} 0&\text{if \;}0\le \rho\le \epsilon/2\,,\\ \frac12\(1+\sin\(\frac{2\pi}\epsilon\,\rho-\frac{3\pi}2\)\) & \text{if \;} \epsilon/2<\rho\le \epsilon\,,\\ 1 & \text{if \;} \epsilon\le \rho\le 1/\epsilon\,,\\ \frac12\(1+\sin\(\rho+\frac\pi{2}-\frac1\epsilon\)\) & \text{if \;} 1/\epsilon\le\rho\le 1/\epsilon+\pi\,,\\ 0 & \text{if \; } \rho\ge 1/\epsilon+\pi\,.\end{cases}
\]
With $\varphi_\star$ defined by~\eqref{varphistar}, the functions $\varphi_\epsilon:=\varphi_\star\,\theta_\epsilon$ are equal to $0$ in a neighborhood of the origin and converge almost everywhere to $\varphi_\star$ as $\epsilon\to0_+$, with
\[
\limsup_{\epsilon\to0_+}\mathcal Q_{\nu,a,0}(\varphi_\epsilon)<+\infty\,.
\]
But these functions are not in $C^\infty\big(\R^2\setminus\{0\},\C\big)$. To end the proof we regularize them using a convolution product.
\end{proof}

An eigenfunction associated with $\lambda_{\nu,a}$ is obtained as a sub-product of our method.
\begin{prop}\label{Prop:qualitative} Let $\nu\in(0,1/2]$. For any $a\in(0,\crit(\nu)]$, the eigenspace of $H_{\nu,a}$ associated with $\lambda_{\nu,a}$ is generated by the spinor
\[
\psi_{\nu,a}(\rho,\theta)=\pmatrix{1\\[4pt]i\,e^{i\theta}\,\sqrt{\tfrac{1-\lambda_{\nu,a}}{1+\lambda_{\nu,a}}}}\,\rho^{\kern-2pt\sqrt{\crit(a)^2-\nu^2}-\frac12}\,e^{-\frac\nu{\crit(a)}\,\rho}\quad\forall\,(\rho,\theta)\in\R^+\times[0,2\pi)\,.
\]
\end{prop}
\begin{proof} The result follows from $\varphi=\varphi_\star$ as in Corollary~\ref{Cor:Thm1a} for the upper component and from~\eqref{EL} for the lower component.\end{proof}

Let us conclude this section by some comments. In Proposition~\ref{Prop:critical-a-HB2} and in Corollary~\ref{Cor:Thm1a}, we have $J(\eta_\star, \nu, a, \lambda_{\nu,a})=0$ and $\mathcal Q_{\nu,a,\lambda_{\nu,a}}(\varphi_\star)=0$ if $a=\crit(\nu)$, but we have to pay attention to the fact that the various terms in the integrals are not all individually integrable. In the proof of Theorem~\ref{Thm:ground-state}, this is reflected by the fact that we have to use a truncation argument. The characterization of the space $\mathcal F_{\nu,a}$ in the critical case $a=\acrit(\nu)$ is more technical than in the subcritical regime and we will not do it here. For similar computations without magnetic field in dimensions $2$ and $3$, see~\cite[Section~1.5 and Appendix~A.3]{Esteban-Lewin-Sere-2019}.

\section{Non-relativistic limit}\label{Sec:NRL}

In this section, we discuss the non-relativistic limit of the ground state spinor and of the ground state energy of the magnetic Dirac-Coulomb operator on $\R^2$ with the Aharonov-Bohm magnetic field $\mathbf A_a$. Let us introduce the \emph{speed of light} $c$ in the operator and consider
\[
H^c_{\nu,a}:=c^2\,\sigma_3-i\,c\,\boldsymbol{\sigma}\cdot\nA-\frac\nu{|x|}\,\mathrm{Id}\,.
\]
Up to this point, we considered atomic units and took $c=1$, $H_{\nu,a}=H^1_{\nu,a}$. Here we consider the limit as $c\to+\infty$.

If $\psi_c$ is an eigenfunction of $H^c_{\nu,a}$ with eigenvalue $\lambda_c$, that is, $H^c_{\nu,a}\,\psi_c=\lambda_c\,\psi_c$, then $\Psi_c(x):=\psi_c(x/c)$ solves $H_{\nu/c,a}\,\Psi_c=\frac{\lambda_c}{c^2}\,\Psi_c$. As a consequence of Proposition~\ref{Prop:qualitative},
\[
\psi_c=\pmatrix{\varphi_c\\\chi_c}\quad\mbox{where}\quad\left\{
\begin{array}{c}
\varphi_c(\rho,\theta)=\rho^{\kern-2pt\sqrt{\crit(a)^2-\nu^2/c^2}-\frac12}\,e^{-\frac\nu{\crit(a)}\,\rho}\\[8pt]
\chi_c(\rho,\theta)=i\,e^{i\theta}\,\sqrt{\tfrac{c^2-\lambda_c}{c^2+\lambda_c}}\,\varphi_c(\rho,\theta)
\end{array}\right.
\;\forall\,(\rho,\theta)\in\R^+\times[0,2\pi)
\]
with $\lambda_c=c^2\,\sqrt{1-\frac{\nu^2}{\crit(a)^2\,c^2}}$
so that, by passing to the limit as $c\to+\infty$, we obtain
\[
\lim_{c\to +\infty}\(\lambda_c-c^2\)=-\,\tfrac{\nu^2}{2\,\crit(a)^2}\,,
\]
and
\[
\varphi_c\to\rho^{-a}\,e^{-\frac\nu{\crit(a)}\,\rho}\,,\quad\chi_c\to0
\]
in $H^1_{\rm loc}\big(\R^2\setminus\{0\},\C\big)$.

We recall that $\crit(a)=\frac12-a$. The eigenvalue problem written as a system is
\[
c\,\D\,\chi_c+\(c^2-\tfrac\nu\rho\,\)\varphi_c=\lambda_c\,\varphi_c\,,\quad c\,\D^*\,\varphi_c-\(c^2+\tfrac\nu\rho\,\)\chi_c=\lambda_c\,\chi_c\,.
\]
After eliminating the lower component using
\[
\chi_c=\frac{c\,\D^*\,\varphi_c}{\lambda_c+c^2+\frac\nu\rho}\,,
\]
we obtain for the upper component the equation
\[
\D\(\frac{c^2\,\D^*\,\varphi_c}{\lambda_c+c^2+\frac\nu\rho}\)-\tfrac\nu\rho\,\varphi_c=\(\lambda_c-c^2\)\varphi_c\,.
\]
and notice that this is consistent with the fact that the limiting solution solves the equation
\[
\D\,\D^*\,\varphi-\tfrac\nu\rho\,\varphi=-\,\tfrac{\nu^2}{2\,\crit(a)^2}\,\varphi
\]
in the sense of distributions. On $C^0(\R^2,\C)\cap C_c^2\big(\R^2\setminus\{0\},\C\big)$, an elementary computation shows that $\D\,\D^* = -\nA^2 - \mathbf B$, where the magnetic field is $\mathbf B=2\pi\,a\,\delta_0$, corresponding to a Dirac mass at the origin. This operator is similar to the Pauli operator for measure valued magnetic fields studied by Erdoes and Vougalter in~\cite{Erdoes-Vougalter} using the Aharonov-Casher transformation~\eqref{varphi-eta}. If we define the quadratic form
\[
q(\varphi, \varphi'):=\int_{\R^2}\dzbar(|x|^a\varphi)\,\overline{\dzbar(|x|^a\varphi')}\,|x|^{-2a}dx\,,
\]
we can follow~\cite[Theorem~2.5]{Erdoes-Vougalter} and define $\D\,\D^*$ as the Friedrichs extension on $L^2(\R^2,\C)$ of the unique self-adjoint operator associated with $q$, with domain
\[
\Big\{ \varphi\in L^2(\R^2, \C)\;:\; q(\varphi, \varphi)<+\infty\;,\; q(\varphi, \cdot)\in L^2(\R^2, \C)'\Big\}\,.
\]

\appendix\section{Appendix}\label{Appendix}

\subsection{The ground state and Laguerre polynomials}\label{Sec:Laguerre}

Based on~\eqref{Fourier} and~\eqref{JFourierDecomposition}, we can provide an alternative computation of the optimal function $\eta_\star$ in Proposition~\ref{Prop:critical-a-HB2}.

As a consequence of the properties of the Laguerre polynomials (see~\cite{Mawhin_2010}), for any $\ell\in\Z$, solutions of
\be{EDO}
-\(\frac{\rho\,\phi'-\ell\,\phi}{(1+\lambda)\,\rho+\nu}\)'-(\ell+1-2\,a)\,\frac{\phi'-\frac\ell\rho\,\phi}{(1+\lambda)\,\rho+\nu}+\(1-\lambda-\frac\nu\rho\)\,\phi=0
\ee
are generated by the functions $\phi(\rho)=\rho^A\,e^{-\,B\,\rho}$ with either
\[
A=a-\frac12-\frac{\crit_\ell(a)}{|\crit_\ell(a)|}\,\sqrt{\crit_\ell(a)^2-\nu^2}\,,\quad B=\frac\nu{\crit_\ell(a)}\,,\quad\lambda=-\,\frac{\sqrt{\crit_\ell(a)^2-\nu^2}}{|\crit_\ell(a)|}\,,
\]
or
\[
A=a-\frac12+\frac{\crit_\ell(a)}{|\crit_\ell(a)|}\,\sqrt{\crit_\ell(a)^2-\nu^2}\,,\quad B=\frac\nu{\crit_\ell(a)}\,,\quad\lambda=\frac{\sqrt{\crit_\ell(a)^2-\nu^2}}{|\crit_\ell(a)|}\,,
\]
where $\crit_\ell(a):=1/2+\ell-a$. However, the integrability of $\rho\mapsto\rho^{2-2a}\,|\phi'(\rho)|^2$ in a positive neighbourhood of $\rho=0$ selects the second one, with $\ell\ge0$. For any $\ell\in\N$, let $\phi_\ell(\rho):=\kappa_\ell\,\rho^{A_\ell(a)}\,e^{-\,B_\ell(a)\,\rho}$ with
\[
A_\ell(a)=a-\frac12+\frac{\crit_\ell(a)}{|\crit_\ell(a)|}\,\sqrt{\crit_\ell(a)^2-\nu^2}\,,\quad B_\ell(a)=\frac\nu{\crit_\ell(a)}\,,\quad\lambda_\ell(a)=\frac{\sqrt{\crit_\ell(a)^2-\nu^2}}{|\crit_\ell(a)|}\,.
\]
With this choice, we find that $J(\eta_\star,\nu,a,\lambda_0(a))\ge0$ with equality if and only if $\kappa_\ell=0$ for any $\ell\ge1$. On the other hand, $\eta_\star$ is a critical point of $J$, so that $\phi_\ell$ solves~\eqref{EDO} with $\lambda=\lambda_\star(\eta_\star,\nu,a)$. Altogether we conclude that $\lambda_\star(\eta_\star,\nu,a)=\lambda_0(a)$.

\subsection{Special cases of Hardy-type inequalities}\label{Sec:SpecialHardy}

Here we list some special cases of Hardy-type inequalities related with the Dirac-Coulomb operator, with or without magnetic fields, which are of interest by themselves. This list of inequalities complements the inequalities of Section~\ref{Sec:Introduction} (after the statement of Proposition~\ref{Prop:critical-a-HB2}).

For any $\eta\in C^\infty_{\rm c}\big(\R^2\setminus\{0\},\C\big)$, Inequality~\eqref{reducedDirac-Hardy} becomes
\[\label{inhom-cartesian-general-CD-acrit}
\int_{\R^2}\(\frac{2\,|x|}{2\,|x|+1}\,\big|(\partial_1+i\,\partial_2)\,\eta\big|^2+|\eta|^2\)dx\ge\frac12\int_{\R^2}\frac{|\eta|^2}{|x|}\,dx
\]
for $\nu=1/2$ and $a=0$, and
\[\label{inhom-cartesian-CD-acrit}
\int_{\R^2}\(\frac{|x|}{|x|+\nu}\,\big|(\partial_1+i\,\partial_2)\,\eta\big|^2+|\eta|^2\)|x|^{2\nu-1}\,dx\ge\nu\int_{\R^2}|\eta|^2\,|x|^{2\nu-2}\,dx
\]
for $\nu\in(0, 1/2)$, $a=\crit(\nu)>0$ while, in that case, a scaling also shows that
\[\label{Magn-Hardy-Pauli-acrit-hom}
\frac1\nu\int_{\R^2}|x|^{2\nu}\,\big|(\partial_1+i\,\partial_2)\eta\big|^2\,|x|^{2\nu}\,dx\ge\nu\int_{\R^2}|\eta|^2\,|x|^{2\nu-2}\,dx\quad\forall\,\eta\in C^\infty_{\rm c}\big(\R^2\setminus\{0\},\C\big)\,.
\]
In the case $a=\crit(\nu)$ and $\nu\in(0, 1/2)$,~\eqref{HardyOpt} becomes
\[
\int_{\R^2}\(\frac{|x|}{|x|+\nu}\,|D_{\crit(\nu)}^*\varphi|^2+|\varphi|^2\)\,dx\ge\nu\int_{\R^2}\frac{|\varphi|^2}{|x|^2}\,dx\quad\forall\,\varphi\in C^\infty_{\rm c}\big(\R^2\setminus\{0\},\C\big)\,.
\]
and it is interesting to relate this last inequality with the homogeneous case of~\eqref{Magn-Hardy-Pauli-intro}, which can be written as
\[\label{Magn-Hardy-Pauli-acrit}
\int_{\R^2}\frac{|x|}\nu\,\big|\big(\sigma\cdot\nabla_{\crit(\nu)}\big)\psi\big|^2\, dx \ge \nu \int_{\R^2} \frac{ |\psi|^2}{|x|}\, dx\quad\forall\,\psi\in C^\infty_{\rm c}(\R^2\setminus\{0\},\C^2)\,.
\]

\subsection{The positron case}\label{Sec:Positron}

In the positron case with positively charged singularity, the eigenvalue problem
\[
H_{a,-\nu}\,\psi=\lambda\,\psi
\]
is transformed using~\eqref{psi-eta} into the system
\begin{align*}
&\eta_1-2\,i\,\rho^{2a}\,\dz\eta_2+\tfrac\nu\rho\,\eta_1=\lambda\,\eta_1\,,\\
&-2\,i\,\rho^{-2a}\,\dzbar\eta_1+\tfrac\nu\rho\,\eta_2-\eta_2=\lambda\,\eta_2\,.
\end{align*}
Using the first equation, we can eliminate the upper component
\[
\eta_1=-\,2\,i\,\frac{\rho^{2a}\,\dz\eta_2}{\lambda-1-\tfrac\nu\rho}
\]
and obtain the equation
\[
-\,4\,\rho^{-2a}\,\dzbar\(\frac{\rho^{2a}\,\dz\eta_2}{\lambda-1-\tfrac\nu\rho}\)-\(1+\lambda-\tfrac\nu\rho\)\eta_2=0\,,
\]
which is a critical point of
\[
J^{(+)}(\eta,\nu,a,\lambda):=\int_{\R^2}\(\frac{4\,|\dz\eta|^2}{1-\lambda+\tfrac\nu{|x|}}+\Big(1+\lambda-\tfrac\nu{|x|}\Big)\,|\eta|^2\)\,|x|^{2a}\,dx\,.
\]
As a function of $\lambda$, $J^{(+)}(\eta,\nu,a,\lambda)$ is monotone increasing and the same method as in the proof of Proposition~\ref{Prop:critical-a-HB2} applies up to the change $\lambda\mapsto-\lambda$. This is why in the computation of the roots of
\[
\nu\,\mu^2-(1-2\,a)\,\mu+\nu=0\,,
\]
we choose $\mu_-=-\(\crit(a)+\sqrt{\crit(a)^2-\nu^2}\)/\nu$ and obtain by solving $(1+\lambda)/(1-\lambda)=\mu_-^2$ that the optimal value for $\lambda$ is
\[
\lambda=-\,\frac{\sqrt{\crit(a)^2-\nu^2}}{\crit(a)}=-\lambda_{\nu,a}\,.
\]
The result of Theorem~\ref{Thm:Hardy} becomes
\[
\int_{\R^2}\(\frac{|\D\chi|^2}{1+\lambda_{\nu,a}+\frac\nu{|x|}}+\Big(1-\lambda_{\nu,a}-\tfrac\nu{|x|}\Big)\,|\chi|^2\)\,dx\ge0\quad\forall\,\chi\in C^\infty_{\rm c}\big(\R^2\setminus\{0\},\C\big)
\]
with equality if $\chi$ is the lower component of $\psi_{-\nu,a}$ in Proposition~\ref{Prop:qualitative}.
This also means that~\eqref{reducedDirac-Hardy} is replaced by
\[
\int_{\R^2}\(\frac{|\dz\eta|^2}{1+\lambda_{\nu,a}+\frac\nu{|x|}}+\Big(1-\lambda_{\nu,a}-\tfrac\nu{|x|}\Big)\,|\eta|^2\)|x|^{2a}\,dx\ge0\quad\forall\,\eta\in C^\infty_{\rm c}\big(\R^2\setminus\{0\},\C\big)
\]
with same consequences: under the assumption $a\in(0,1/2]$ and $\nu\in(0,\crit(a)]$, we obtain
\[
\int_{\R^2}\(\frac4\nu\,|x|\,|\dz\eta|^2+2\,|\eta|^2\)\,|x|^{2a}\,dx\ge\nu\int_{\R^2}|\eta|^2\,|x|^{2a-1}\,dx\quad\forall\,\eta\in C^\infty_{\rm c}\big(\R^2\setminus\{0\},\C\big)
\]
and
\[
\int_{\R^2}|\dz\eta|^2\,|x|^{2a+1}\,dx\ge\frac14\,\crit(a)^2\int_{\R^2}|\eta|^2\,|x|^{2a-1}\,dx\quad\forall\,\eta\in C^\infty_{\rm c}\big(\R^2\setminus\{0\},\C\big)\,.
\]

\subsection*{Acknowledgments}

\noindent{\small This research has been partially supported by the projects \emph{EFI}, contract~ANR-17-CE40-0030 (J.D.) and \emph{molQED} (M.J.E.) of the French National Research Agency (ANR) and by the NSF grant DMS-1856645 (M.L.). \\[4pt]
\small\copyright\,2020 by the authors. This paper may be reproduced, in its entirety, for non-commercial purposes.}

\small
\bibliographystyle{siam}\bibliography{Ari70-proc-DEL}

\begin{thebibliography}{10}

\bibitem{PhysRevA.19.2461}
{\sc Y.~Aharonov and A.~Casher}, {\em Ground state of a spin-\textonehalf{}
  charged particle in a two-dimensional magnetic field}, Phys. Rev. A, 19
  (1979), pp.~2461--2462.

\bibitem{Arai_1983}
{\sc M.~Arai}, {\em On essential selfadjointness, distinguished selfadjoint
  extension and essential spectrum of {D}irac operators with matrix valued
  potentials}, Publications of the Research Institute for Mathematical
  Sciences, 19 (1983), pp.~33--57.

\bibitem{Arai_1982}
{\sc M.~Arai and O.~Yamada}, {\em Essential selfadjointness and invariance of
  the essential spectrum for {D}irac operators}, Publications of the Research
  Institute for Mathematical Sciences, 18 (1982), pp.~973--985.

\bibitem{Arrizabalaga_2011}
{\sc N.~Arrizabalaga}, {\em Distinguished self-adjoint extensions of {D}irac
  operators via {Hardy-Dirac} inequalities}, Journal of Mathematical Physics,
  52 (2011), p.~092301.

\bibitem{bonheureetal}
{\sc D.~Bonheure, J.~Dolbeault, M.~J. Esteban, A.~Laptev, and M.~Loss}, {\em
  Symmetry results in two-dimensional inequalities for
  {A}haronov{\textendash}{B}ohm magnetic fields}, Communications in
  Mathematical Physics, 375 (2019), pp.~2071--2087.

\bibitem{MR2379440}
{\sc R.~Bosi, J.~Dolbeault, and M.~J. Esteban}, {\em Estimates for the optimal
  constants in multipolar {H}ardy inequalities for {S}chr\"odinger and {D}irac
  operators}, Commun. Pure Appl. Anal., 7 (2008), pp.~533--562.

\bibitem{MR1261368}
{\sc J.~F. Brasche and H.~Neidhardt}, {\em Some remarks on {K}re\u{\i}n's
  extension theory}, Math. Nachr., 165 (1994), pp.~159--181.

\bibitem{Burnap_1981}
{\sc C.~Burnap, H.~Brysk, and P.~F. Zweifel}, {\em {Dirac Hamiltonian for
  strong Coulomb fields}}, Il Nuovo Cimento B, 64 (1981), pp.~407--419.

\bibitem{Cacciafesta_Fanelli_2017}
{\sc F.~Cacciafesta and L.~Fanelli}, {\em Dispersive estimates for the {D}irac
  equation in an {A}haronov--{B}ohm field}, Journal of Differential Equations,
  263 (2017), pp.~4382--4399.

\bibitem{PhysRevA.72.022103}
{\sc A.~C\'aceres and C.~Doran}, {\em Minimax determination of the energy
  spectrum of the {D}irac equation in a {S}chwarzschild background}, Phys. Rev.
  A, 72 (2005), p.~022103.

\bibitem{RevModPhys.81.109}
{\sc A.~H. Castro~Neto, F.~Guinea, N.~M.~R. Peres, K.~S. Novoselov, and A.~K.
  Geim}, {\em The electronic properties of graphene}, Rev. Mod. Phys., 81
  (2009), pp.~109--162.

\bibitem{1928}
{\sc C.~G. Darwin}, {\em The wave equations of the electron}, Proceedings of
  the Royal Society of London. Series A, Containing Papers of a Mathematical
  and Physical Character, 118 (1928), pp.~654--680.

\bibitem{Datta_1988}
{\sc S.~N. Datta and G.~Devaiah}, {\em The minimax technique in relativistic
  {Hartree-Fock} calculations}, Pramana, 30 (1988), pp.~387--405.

\bibitem{MR3962850}
{\sc D.-A. Deckert and M.~Oelker}, {\em Distinguished self-adjoint extension of
  the two-body {D}irac operator with {C}oulomb interaction}, Ann. Henri
  Poincar\'{e}, 20 (2019), pp.~2407--2445.

\bibitem{DDEV}
{\sc J.~Dolbeault, M.~J. Esteban, J.~Duoandikoetxea, and L.~Vega}, {\em
  {H}ardy-type estimates for {D}irac operators}, Annales Scientifiques de
  l'{\'E}cole Normale Sup{\'e}rieure, 40 (2007), pp.~885--900.

\bibitem{Dolbeault-Esteban-Loss-06}
{\sc J.~Dolbeault, M.~J. Esteban, and M.~Loss}, {\em {Relativistic hydrogenic
  atoms in strong magnetic fields}}, {Annales Henri Poincar{\'e}}, 8 (2007),
  pp.~749--779.

\bibitem{MR2091354}
{\sc J.~Dolbeault, M.~J. Esteban, M.~Loss, and L.~Vega}, {\em An analytical
  proof of {H}ardy-like inequalities related to the {D}irac operator}, J.
  Funct. Anal., 216 (2004), pp.~1--21.

\bibitem{DolEstSer-00}
{\sc J.~Dolbeault, M.~J. Esteban, and E.~S\'er\'e}, {\em On the eigenvalues of
  operators with gaps. application to {D}irac operators}, {Journal of
  Functional Analysis}, 174 (2000), pp.~208--226.

\bibitem{MR1767717}
\leavevmode\vrule height 2pt depth -1.6pt width 23pt, {\em Variational
  characterization for eigenvalues of {D}irac operators}, Calc. Var. Partial
  Differential Equations, 10 (2000), pp.~321--347.

\bibitem{dolbeault2003variational}
\leavevmode\vrule height 2pt depth -1.6pt width 23pt, {\em A variational method
  for relativistic computations in atomic and molecular physics}, International
  journal of quantum chemistry, 93 (2003), pp.~149--155.

\bibitem{MR2239275}
\leavevmode\vrule height 2pt depth -1.6pt width 23pt, {\em General results on
  the eigenvalues of operators with gaps, arising from both ends of the gaps.
  {A}pplication to {D}irac operators}, J. Eur. Math. Soc. (JEMS), 8 (2006),
  pp.~243--251.

\bibitem{PhysRevLett.85.4020}
{\sc J.~Dolbeault, M.~J. Esteban, E.~S\'er\'e, and M.~Vanbreugel}, {\em
  Minimization methods for the one-particle {D}irac equation}, Phys. Rev.
  Lett., 85 (2000), pp.~4020--4023.

\bibitem{Dong-Ma_2003}
{\sc S.-H. Dong and Z.-Q. Ma}, {\em Exact solutions to the {D}irac equation
  with a {C}oulomb potential in {$2+1$} dimensions}, Physics Letters A, 312
  (2003), pp.~78--83.

\bibitem{Erdoes-Vougalter}
{\sc L.~Erd\H{o}s and V.~Vougalter}, {\em Pauli operator and
  {A}haronov-{C}asher theorem for measure valued magnetic fields}, Comm. Math.
  Phys., 225 (2002), pp.~399--421.

\bibitem{MR2602013}
{\sc M.~J. Esteban, M.~Lewin, and A.~Savin}, {\em Symmetry breaking of
  relativistic multiconfiguration methods in the nonrelativistic limit},
  Nonlinearity, 23 (2010), pp.~767--791.

\bibitem{Esteban-Lewin-Sere-2019}
{\sc M.~J. Esteban, M.~Lewin, and E.~S\'{e}r\'{e}}, {\em Domains for
  {D}irac-{C}oulomb min-max levels}, Rev. Mat. Iberoam., 35 (2019),
  pp.~877--924.

\bibitem{Esteban-Loss-1}
{\sc M.~J. Esteban and M.~Loss}, {\em Self-adjointness for {D}irac operators
  via {H}ardy-{D}irac inequalities}, J. Math. Phys., 48 (2007), pp.~112107, 8.

\bibitem{Esteban-Loss-2}
\leavevmode\vrule height 2pt depth -1.6pt width 23pt, {\em Self-adjointness via
  partial {H}ardy-like inequalities}, in Mathematical results in quantum
  mechanics, World Sci. Publ., Hackensack, NJ, 2008, pp.~41--47.

\bibitem{MR1479240}
{\sc M.~J. Esteban and E.~S\'er\'e}, {\em Existence and multiplicity of
  solutions for linear and nonlinear {D}irac problems}, in Partial differential
  equations and their applications ({T}oronto, {ON}, 1995), vol.~12 of CRM
  Proc. Lecture Notes, Amer. Math. Soc., Providence, RI, 1997, pp.~107--118.

\bibitem{MR1869528}
\leavevmode\vrule height 2pt depth -1.6pt width 23pt, {\em Nonrelativistic
  limit of the {D}irac-{F}ock equations}, Ann. Henri Poincar\'{e}, 2 (2001),
  pp.~941--961.

\bibitem{Falomir_2001}
{\sc H.~Falomir and P.~A.~G. Pisani}, {\em {H}amiltonian self-adjoint
  extensions for {$(2+1)$}-dimensional {D}irac particles}, Journal of Physics
  A: Mathematical and General, 34 (2001), pp.~4143--4154.

\bibitem{MR3817548}
{\sc M.~Gallone and A.~Michelangeli}, {\em Discrete spectra for critical
  {D}irac-{C}oulomb {H}amiltonians}, J. Math. Phys., 59 (2018), pp.~062108, 19.

\bibitem{MR3933559}
\leavevmode\vrule height 2pt depth -1.6pt width 23pt, {\em Self-adjoint
  realisations of the {D}irac-{C}ou\-lomb {H}amiltonian for heavy nuclei},
  Anal. Math. Phys., 9 (2019), pp.~585--616.

\bibitem{MR1870409}
{\sc V.~Georgescu and M.~M\u{a}ntoiu}, {\em On the spectral theory of singular
  {D}irac type {H}amiltonians}, J. Operator Theory, 46 (2001), pp.~289--321.

\bibitem{gordon1928energieniveaus}
{\sc W.~Gordon}, {\em Die {E}nergieniveaus des {W}asserstoffatoms nach der
  {D}iracschen {Q}uantentheorie des {E}lektrons}, Zeitschrift f{\"u}r Physik,
  48 (1928), pp.~11--14.

\bibitem{MR1724845}
{\sc M.~Griesemer and H.~Siedentop}, {\em A minimax principle for the
  eigenvalues in spectral gaps}, J. London Math. Soc. (2), 60 (1999),
  pp.~490--500.

\bibitem{Hogreve_2012}
{\sc H.~Hogreve}, {\em The overcritical {Dirac--Coulomb} operator}, Journal of
  Physics A: Mathematical and Theoretical, 46 (2012), p.~025301.

\bibitem{10.1007/BFb0067087}
{\sc H.~Kalf, U.-W. Schmincke, J.~Walter, and R.~W{\"u}st}, {\em {On the
  spectral theory of {S}chr{\"o}dinger and {D}irac operators with strongly
  singular potentials}}, in Spectral Theory and Differential Equations, W.~N.
  Everitt, ed., Berlin, Heidelberg, 1975, Springer Berlin Heidelberg,
  pp.~182--226.

\bibitem{Klaus_1979}
{\sc M.~Klaus and R.~W{\"u}st}, {\em Characterization and uniqueness of
  distinguished self-adjoint extensions of {D}irac operators}, Communications
  in Mathematical Physics, 64 (1979), pp.~171--176.

\bibitem{Klaus_1979b}
{\sc M.~Klaus and R.~W{\"u}st}, {\em Spectral properties of {D}irac operators
  with singular potentials}, Journal of Mathematical Analysis and Applications,
  72 (1979), pp.~206--214.

\bibitem{MR0024574}
{\sc M.~Krein}, {\em The theory of self-adjoint extensions of semi-bounded
  {H}ermitian transformations and its applications. {I}}, Rec. Math. [Mat.
  Sbornik] N.S., 20(62) (1947), pp.~431--495.

\bibitem{Kullie_2004}
{\sc O.~Kullie, D.~Kolb, and A.~Rutkowski}, {\em Two-spinor fully relativistic
  finite-element ({FEM}) solution of the two-center {C}oulomb problem},
  Chemical Physics Letters, 383 (2004), pp.~215--221.

\bibitem{Laptev-Weidl}
{\sc A.~Laptev and T.~Weidl}, {\em {H}ardy inequalities for magnetic
  {D}irichlet forms}, in Mathematical results in quantum mechanics ({P}rague,
  1998), vol.~108 of Oper. Theory Adv. Appl., Birkh\"{a}user, Basel, 1999,
  pp.~299--305.

\bibitem{Mawhin_2010}
{\sc J.~Mawhin and A.~Ronveaux}, {\em {Schr{\"o}dinger and {D}irac equations
  for the hydrogen atom, and {L}aguerre polynomials}}, Archive for History of
  Exact Sciences, 64 (2010), pp.~429--460.

\bibitem{MorMul-15}
{\sc S.~Morozov and D.~M{\"u}ller}, {\em On the minimax principle for
  {C}oulomb--{D}irac operators}, Mathematische Zeitschrift, 280 (2015),
  pp.~733--747.

\bibitem{Mueller-2016}
{\sc D.~M\"{u}ller}, {\em Minimax principles, {H}ardy-{D}irac inequalities, and
  operator cores for two and three dimensional {C}oulomb-{D}irac operators},
  Doc. Math., 21 (2016), pp.~1151--1169.

\bibitem{Nenciu_1976}
{\sc G.~Nenciu}, {\em Self-adjointness and invariance of the essential spectrum
  for {D}irac operators defined as quadratic forms}, Communications in
  Mathematical Physics, 48 (1976), pp.~235--247.

\bibitem{MR462346}
{\sc G.~Nenciu}, {\em Distinguished self-adjoint extension for {D}irac operator
  with potential dominated by multicenter {C}oulomb potentials}, Helv. Phys.
  Acta, 50 (1977), pp.~\hbox{1--3}.

\bibitem{Ounaies-2000}
{\sc H.~Ounaies}, {\em Perturbation method for a class of nonlinear {D}irac
  equations}, Differential Integral Equations, 13 (2000), pp.~707--720.

\bibitem{SchSolTok-20}
{\sc L.~Schimmer, J.~P. Solovej, and S.~Tokus}, {\em {F}riedrichs extension and
  min--max principle for operators with a gap}, Ann. Henri Poincar\'{e}, 21
  (2020), pp.~327--357.

\bibitem{Schmincke_1972b}
{\sc U.-W. Schmincke}, {\em Distinguished selfadjoint extensions of {D}irac
  operators}, Mathematische Zeitschrift, 129 (1972), pp.~335--349.

\bibitem{Schmincke_1972}
\leavevmode\vrule height 2pt depth -1.6pt width 23pt, {\em Essential
  selfadjointness of {D}irac operators with a strongly singular potential},
  Mathematische Zeitschrift, 126 (1972), pp.~71--81.

\bibitem{Schwabl_2004}
{\sc F.~Schwabl}, {\em Symmetries and further properties of the {D}irac
  equation}, in Advanced Texts in Physics, Springer Berlin Heidelberg, 2004,
  pp.~209--244.

\bibitem{Sitenko_2000}
{\sc Y.~A. Sitenko}, {\em Self-adjointness of the two-dimensional massless
  {D}irac {H}amiltonian and vacuum polarization effects in the background of a
  singular magnetic vortex}, Annals of Physics, 282 (2000), pp.~167--217.

\bibitem{PhysRevLett.57.1091}
{\sc J.~D. Talman}, {\em Minimax principle for the {D}irac equation}, Phys.
  Rev. Lett., 57 (1986), pp.~1091--1094.

\bibitem{Thaller-1992}
{\sc B.~Thaller}, {\em The {D}irac equation}, Texts and Monographs in Physics,
  Springer-Verlag, Berlin, 1992.

\bibitem{Lenthe_1993}
{\sc E.~van Lenthe, E.~Baerends, and J.~Snijders}, {\em Relativistic regular
  two-component {H}amiltonians}, The Journal of Chemical Physics, 99 (1993),
  pp.~4597--4610.

\bibitem{Voronov_2007}
{\sc B.~L. Voronov, D.~M. Gitman, and I.~V. Tyutin}, {\em The {D}irac
  {H}amiltonian with a superstrong {C}oulomb field}, Theoretical and
  Mathematical Physics, 150 (2007), pp.~34--72.

\bibitem{Vozmediano_2010}
{\sc M.~Vozmediano, M.~Katsnelson, and F.~Guinea}, {\em Gauge fields in
  graphene}, Physics Reports, 496 (2010), pp.~109--148.

\bibitem{W_st_1973}
{\sc R.~W{\"u}st}, {\em A convergence theorem for selfadjoint operators
  applicable to {D}irac operators with cutoff potentials}, Mathematische
  Zeitschrift, 131 (1973), pp.~339--349.

\bibitem{W_st_1975}
\leavevmode\vrule height 2pt depth -1.6pt width 23pt, {\em Distinguished
  self-adjoint extensions of {D}irac operators constructed by means of cut-off
  potentials}, Mathematische Zeitschrift, 141 (1975), pp.~93--98.

\bibitem{W_st_1977}
\leavevmode\vrule height 2pt depth -1.6pt width 23pt, {\em {D}irac operations
  with strongly singular potentials}, Mathematische Zeit\-schrift, 152 (1977),
  pp.~259--271.

\bibitem{MR1451618}
{\sc J.~Xia}, {\em On the contribution of the {C}oulomb singularity of
  arbitrary charge to the {D}irac {H}amiltonian}, Trans. Amer. Math. Soc., 351
  (1999), pp.~1989--2023.

\bibitem{Zhang_2004}
{\sc H.~Zhang, O.~Kullie, and D.~Kolb}, {\em Minimax {LCAO} approach to the
  relativistic two-centre {C}oulomb problem and its finite element ({FEM})
  spectrum}, Journal of Physics B: Atomic, Molecular and Optical Physics, 37
  (2004), pp.~905--916.

\end{thebibliography}
\end{document}